\documentclass[11pt,final]{article}
\usepackage{amsmath,amsthm,amsfonts,amssymb,graphicx,enumerate,psfrag}
\usepackage{mathrsfs}
\usepackage{fullpage}

\usepackage[colorlinks,citecolor=blue,urlcolor=blue]{hyperref}
\usepackage[utf8]{inputenc} 

\newtheorem{lemma}{Lemma}

\newtheorem{remark}{Remark}

\newtheorem{theorem}[lemma]{Theorem}

\newtheorem{corollary}[lemma]{Corollary}

\newcommand{\EE}{{\mathbb{E}}}
\newcommand{\cov}{{\mathrm{Cov}}}
\newcommand{\VV}{{\mathbb{V}}}
\newcommand{\GG}{{\mathbb{G}}}
\newcommand{\PP}{{\mathbb{P}}}

\def\rP{{\mathscr{P}}}

\def\rX{{\mathscr{X}}}

\newcommand{\dZ}{\mathbb {Z}}
\newcommand{\dN}{\mathbb {N}}
\newcommand{\dR}{\mathbb {R}}

\newcommand{\cB}{\mathscr {B}}

\newcommand{\cN}{\mathcal {N}}
\newcommand{\cO}{\mathcal {O}}

\newcommand{\bX}{{\mathbf{X}}}
\newcommand{\bY}{{\mathbf{Y}}}
\newcommand{\bZ}{{\mathbf{Z}}}

\newcommand{\zz}{\mathfrak z}

\newcommand{\tmix}{{\rm t}_{\textsc{mix}}}

\newcommand{\dkl}{\mathrm{d}_{\textsc{kl}}}

\newcommand{\rL}{\mathscr{L}}

\newcommand{\sep}{\mathrm{sep}}
\title{\vspace{-.9cm}Universality of cutoff for exclusion with reservoirs}
\author{Justin Salez\footnote{CEREMADE, CNRS, UMR 7534, Université Paris-Dauphine, PSL University, 75016 Paris, France}}
\begin{document}
\maketitle
\begin{abstract}
We consider the reversible exclusion process with reservoirs on arbitrary networks. We  characterize the spectral gap, mixing time, and mixing window of the process, in terms of certain simple statistics of the underlying network. Among other consequences, we establish a non-conservative analogue of Aldous's spectral gap conjecture, and we show that cutoff occurs if and only if the product condition is satisfied. We illustrate this by providing explicit cutoffs  on  discrete lattices of arbitrary dimensions and boundary conditions, which substantially generalize recent one-dimensional results. We also obtain cutoff phenomena in relative entropy, Hilbert norm, separation distance and supremum norm. Our proof exploits  negative dependence in a novel, simple way to reduce the understanding of the whole process to that of single-site marginals. We believe that this approach will find other applications.
\end{abstract}
\tableofcontents
 \vspace{.5cm}
\section{Introduction}
The exclusion process \cite{MR268959,MR2108619} is a classical model of interacting random walks in which indistinguishable particles attempt to evolve independently on a graph, except that their jumps are canceled if the destination is already occupied. 
 Here we consider the non-conservative variant of the model, where particles may additionally be created and annihilated at certain vertices, modeling contact with an external {reservoir}. We refer the reader to the papers \cite{MR1997915,MR2437527} and the more recent works \cite{gantert2021mixing,goncalves2021sharp} for motivations and background on this process.
 
\subsection{Setup}

\paragraph{Networks.} Let us first specify the geometry of the model  and introduce some terminology.  Throughout the paper, we consider a \emph{network}  $G=(V,c,\kappa)$ consisting of
\begin{itemize}
\item a finite set $V$ whose elements are called \emph{vertices};
\item a symmetric array $c\colon V\times V\to\dR_+$, whose entries are called \emph{conductances};
\item a function $\kappa\colon V\to\dR_+$, whose entries are called \emph{external rates}.
\end{itemize}
The support of $c(\cdot,\cdot)$ constitutes the set of \emph{edges} of the network, along which particles can move. The support of $\kappa(\cdot)$ represents the \emph{boundary} of the network, where particles can be created or annihilated due to  contact with an external \emph{reservoir}. To ensure irreducibility, we will always assume that the network $G$ is connected, and that its boundary is not empty. A simple way to produce networks consists in ``cutting them out'' from some fixed, locally finite graph $\GG=(\VV,\EE)$. More precisely, we choose a finite connected subset $V\subset \VV$ and, for $i,j\in V$, we set
\begin{eqnarray*}
c(i,j) \ := \ {\bf 1}_{\{i,j\}\in \EE} & \textrm{ and } & 
\kappa(i) \ := \ \sum_{k\in\VV\setminus V}{\bf 1}_{\{i,k\}\in\EE}.
\end{eqnarray*}
We call $G=(V,c,\kappa)$ the \emph{network induced by $V$ in $\GG$}. An important example to have in mind, and to which we shall come back later, is the network induced by an hypercube $V=[n]^d$ in the $d-$dimensional square lattice $\GG=\dZ^d$. 

\paragraph{State space and generator.}As usual, the state of the system is represented by a binary vector $x=(x_i)_{i\in V}$, with $x_i=1$ if the  vertex $i$ is occupied and $x_i=0$ if it is empty.   We let $x^{i\leftrightarrow j}$, $x^{i,1}$ and $x^{i,0}$ denote the vectors obtained from $x$ by respectively swapping the $i-$th and $j-$th coordinates (exchange), resetting the $i-$th coordinate to $1$ (creation), and resetting it to  $0$ (annihilation). With this notation, the \emph{exclusion process with reservoir density  $\rho\in(0,1)$ on the network $G$} is the continuous-time Markov chain with state space $\rX=\{0,1\}^V$ and whose infinitesimal generator  $\rL$ acts on observables $f\colon \rX\to\dR$ as follows:
\begin{eqnarray*}
(\rL f)(x) & := & \frac{1}{2}\sum_{i,j\in V}c(i,j)\left[f(x^{i\leftrightarrow j})-f(x)\right]+\sum_{i\in V}\kappa(i)\left[\rho f(x^{i,1})+(1-\rho)f(x^{i,0})-f(x)\right].
\end{eqnarray*}
In words, every pair of vertices $\{i,j\}$ exchange contents at rate $c(i,j)$, and every vertex $i$ resamples its content afresh  at rate $\kappa(i)$ according to  $\cB_\rho$, the Bernoulli distribution with mean $\rho$.

\paragraph{Convergence to equilibrium.}
The generator $\rL$  is clearly irreducible and reversible w.r.t. the product measure $\pi=\cB_\rho^{\otimes V}$. As a consequence, the associated semi-group $(\rP_t)_{t\ge 0}$  \emph{mixes}:
\begin{eqnarray*}
\forall x,y\in\rX,\quad \rP_t(x,y) & \xrightarrow[t\to\infty]{} & \pi(y).
\end{eqnarray*}
The time-scale on  which this convergence occurs is  measured by the  so-called \emph{mixing times} 
\begin{eqnarray}
\label{def:mixing}
\tmix(\varepsilon) & := & \min\left\{t\ge 0\colon \max_{x\in\rX}\|\rP_t(x,\cdot)-\pi\|_{\textsc{tv}}\le \varepsilon\right\},\quad \varepsilon\in(0,1),
\end{eqnarray}
where $\|\mu-\pi\|_{\textsc{tv}}=\max_{A\subseteq\rX}|\mu(A)-\pi(A)|$ denotes the total-variation distance between two probability measures $\mu,\pi$ on $\rX$. The following stronger measures of discrepancy will also be considered:
\begin{itemize}
\item Separation distance: $\sep(\mu,\pi):=\max_{x\in\rX}\left(1-\frac{\mu(x)}{\pi(x)}\right)$;
\item Relative entropy: $
\dkl(\mu||\pi) :=  \sum_{x\in\rX}\mu(x)\log\frac{\mu(x)}{\pi(x)}$;
\item Hilbert norm: $
\|\frac{\mu}{\pi}-1\|_{L^2_\pi}^2 :=  {\sum_{x\in\rX}\pi(x)\left(\frac{\mu(x)}{\pi(x)}-1\right)^2}$;
\item Supremum norm : $
\|\frac{\mu}{\pi}-1\|_{\infty} :=  \max_{x\in\rX}\left|\frac{\mu(x)}{\pi(x)}-1\right|$.
\end{itemize}
We refer the reader to the books \cite{MR2341319,MR3726904}) for details on those natural quantities. Understanding how the  fundamental parameter $\tmix(\varepsilon)$ depends on the size and geometry of the underlying network constitutes a natural and important question, to which the present paper is devoted. 

\paragraph{State of the art.} While the convergence to equilibrium of the exclusion process \emph{without} reservoirs has received a considerable attention (see, e.g., \cite{MR2023023,MR2629990,MR2244427,MR3077529,MR3474475,MR3551201,MR3689972,MR4164852,
MR4164461}), only very little has been said about the non-conservative version of the model. In fact, there seems to be only two examples of networks on which the mixing time of the exclusion process with reservoirs is known. The first is the segment of length $n$ with a reservoir at one end ($V=[n]$, $c(i,j)={\bf 1}_{|i-j|=1}$, $\kappa(i)={\bf 1}_{i=n}$), for which Gantert, Nestoridi and Schmid  \cite{gantert2021mixing} recently established the estimate
\begin{eqnarray}
\label{result:1}
\tmix(\varepsilon) & = & \frac{2n^2\log n}{\pi^2}+o(n^2\log n),
\end{eqnarray}
as $n\to\infty$,  for any fixed $\varepsilon,\rho\in(0,1)$. The proof crucially relies on a coupling used by Lacoin for the conservative version of the problem \cite{MR3474475}, which seems specific to the one-dimensional setup. The second example is the variant of the above network where reservoirs are present at both ends of the segment ($\kappa(i)={\bf 1}_{i=1}+{\bf 1}_{i=n}$). Using a new and promising application of Yau's celebrated relative entropy method \cite{MR1121850},  Gonçalves, Jara, Marinho and Menezes \cite{goncalves2021sharp}  proved that
\begin{eqnarray}
\label{result:2}
\tmix(\varepsilon) & = & \frac{n^2\log n}{2\pi^2}+cn^2+o(n^2),
\end{eqnarray}
as $n\to\infty$, where $c=c(\varepsilon,\rho)\in\dR$ is an explicit constant.  In both cases,  the remarkable fact that the leading order term does not depend on  $\varepsilon\in(0,1)$ reflects a sharp transition to equilibrium known as a \emph{cutoff}, see \cite{MR1374011} for an introduction. We emphasize that the works \cite{gantert2021mixing,goncalves2021sharp} are not limited to the above estimates: \cite{gantert2021mixing} initiates the more delicate study of the non-reversible setup, where the reservoir densities differ at the two end-points of the segment, while \cite{goncalves2021sharp} provides a very detailed picture of the  system started from any smooth initial condition. To the best of our knowledge, however, the current understanding of the mixing properties of the exclusion with reservoirs is limited to the two aforementioned one-dimensional networks. 

\paragraph{Our contribution.} In the present paper, we consider the exclusion process on an \emph{arbitrary} network $G$. In this level of generality, we determine the spectral gap, mixing time, and mixing window of the process, in terms of certain simple spectral statistics of $G$. As a by-product,  we completely characterize the occurrence of the cutoff phenomenon, and obtain multi-dimensional generalizations of (\ref{result:1}) and (\ref{result:2}). We also establish cutoff phenomena in relative entropy, Hilbert norm, separation distance and supremum norm. Our proof exploits  negative dependence in a novel, simple way to reduce the understanding of the whole process to that of single-site marginals. We believe that this approach will find other applications.

\subsection{Results}

The main message of our paper is that the mixing properties of the  high-dimensional operator $\rL$ are entirely governed by those of a much lower-dimensional object, namely the  $V\times V$ matrix
\begin{eqnarray*}
\Delta(i,j) & := &
\left\{
\begin{array}{ll}
c(i,j) & \textrm{if }j\ne i\\
-\kappa(i)-\sum_{k\in V}c(i,k) & \textrm{if }j=i.
\end{array}
\right.
\end{eqnarray*}
We call $\Delta$ the \emph{Laplace matrix} of the network $G=(V,c,\kappa)$. It  describes the evolution of a single random walker moving on $V$ according to the conductances $c(\cdot,\cdot)$ and killed at the space-varying rate $\kappa(\cdot)$. Writing $\tau$ for the time at which the walker is killed, we consider the key statistics
\begin{eqnarray*}
\zz_i(t) & := & \PP_i\left(\tau>t\right),
\end{eqnarray*}
where the notation $\PP_i$ indicates that the walk starts at $i$. Note that $\zz$ solves the differential equation
\begin{eqnarray}
\label{death}
\frac{\mathrm d \zz}{\mathrm dt} & = & \Delta \zz,\qquad \zz(0)=\bf 1,
\end{eqnarray}
i.e. $\zz(t)=e^{t\Delta}{\bf 1}$. Our main result asserts that the distance to equilibrium of the exclusion process with reservoirs at any time $t$ is uniformly controlled, in a two-sided way, by the simple quantity
$
 \|\zz(t)\|^2  :=  \sum_{i\in V}\zz_i^2(t).
$
Below and throughout the paper, we use the following short-hands:
\begin{eqnarray*}
\rho_\star & := & \min\left(\rho,1-\rho\right),\qquad x_\star \ := \ \left\{
\begin{array}{ll}
{\bf 1} & \textrm{if }\rho\le \frac{1}{2}\\
\bf 0 & \textrm{else}.
\end{array}
\right.
\end{eqnarray*}
\begin{theorem}[Two-sided estimate]\label{th:main}From the extremal initial state $x_\star$, we have the lower-bound
\begin{eqnarray*}
\left\|{\rP_t(x_\star,\cdot)}-\pi\right\|_{\textsc{tv}} & \ge & \frac{\|\zz(t)\|^2}{4+\|\zz(t)\|^2},
\end{eqnarray*}
at any time $t\ge 0$. 
Conversely,  we have the uniform  Hilbert-norm upper-bound
\begin{eqnarray*}
\max_{x\in\rX}\left\|\frac{\rP_t(x,\cdot)}{\pi}-1\right\|_{L^2_\pi} & \le & \sqrt{\exp\left(\frac{\|\zz(t)\|^2}{\rho_\star}\right)-1}.
\end{eqnarray*}
\end{theorem}
The first estimate asserts that $\|\zz(t)\|$ \emph{needs} to be small in order for the exclusion process to be well-mixed, at least from the extremal initial state $x_\star$. The second (and much more surprising) bound asserts that  this condition actually also \emph{suffices} to guarantee mixing from \emph{any} initial condition, even when the distance to equilibrium is measured in the strong $L^2_\pi-$norm. Thus,   the mixing time of the  process is essentially the time at which $\|\zz(t)\|$ becomes small. In order to study this quantity, observe that  (\ref{death}) implies the spectral expression
\begin{eqnarray}
\label{zz:key}
 \|\zz(t)\|^2 & = & \sum_{k=1}^{|V|}e^{-2\lambda_k t}\langle\psi_k,{\bf 1}\rangle^2,
 \end{eqnarray}
where $0< \lambda_1\le  \ldots\le \lambda_{|V|}$ denote the eigenvalues of the symmetric positive-definite matrix $-\Delta$, and $\psi_1,\ldots,\psi_{|V|}$ a corresponding orthonormal basis of  eigenvectors. It follows that the mixing properties of our interacting particle system are entirely dictated by the spectral statistics of $\Delta$ and, most particularly, by the Perron eigen-pair $(\lambda_1,\psi_1)$, henceforth denoted simply $(\lambda,\psi)$. This has a number of important consequences, which we now enumerate  (see Section \ref{sec:coro} for details).

\paragraph{Spectral gap.} The most fundamental parameter of a reversible Markov generator $\rL$  is arguably its  \emph{spectral gap} or \emph{Poincaré constant}, defined as the second smallest eigenvalue of  $-\rL$ . This constant provides quantitative controls on a variety of properties of the process, including concentration of measure (via Poincaré's inequality), isoperimetry (via Cheeger's inequality) and mixing (via contraction in the $L^2_\pi-$norm), see the books \cite{MR2341319,MR3726904} for details. Our main estimate implies the following pleasant surprise, which does not seem to have been noted before.
\begin{corollary}[Spectral gap]\label{co:gap}The spectral gap of the exclusion process with reservoir density $\rho\in(0,1)$ on $G$ coincides with the smallest eigenvalue $\lambda$ of the  matrix $-\Delta$.
\end{corollary}
Recall that $\rL$ is a $2^{|V|}-$dimensional operator describing the joint evolution of many interacting particles, while $\Delta$ is a  $|V|-$dimensional matrix describing the motion of a single particle. The drastic dimensionality reduction stated in Corollary \ref{co:gap} can be seen as a non-conservative analogue of the celebrated Aldous' \emph{spectral gap conjecture}, now proved by Caputo, Liggett and Richthammer  \cite{MR2629990}. See Hermon and Salez \cite{MR3984254} for a similar conclusion in the case of the Zero-Range process. Corollary \ref{co:gap} can be used to explicitly compute the spectral gap in various examples, see Section \ref{sec:ex}.

\paragraph{Window and cutoff.}A second notable consequence of Theorem \ref{th:main} is the following universal estimate on the \emph{width} of the mixing window, i.e., the time-scale during which the system  moves from being barely mixed to being completely mixed.

\begin{corollary}[Mixing window]\label{co:width}There is a constant $c=c(\varepsilon,\rho)$, not depending on $G$, such that 
\begin{eqnarray*}
\tmix(\varepsilon)-\tmix(1-\varepsilon) & \le & \frac{c}{\lambda}.
\end{eqnarray*}
\end{corollary}
This result implies the following characterization of cutoff for exclusion processes with reservoirs. Here and throughout the paper, when considering a  sequence of networks  $(G_n)_{n\ge 1}$ instead of a fixed network 
$G$, we naturally index all relevant quantities by $n$ (e.g., $|V_n|$, $\tmix^{(n)}$, $\lambda_n$, $\psi_n$).

\begin{corollary}[Characterization of cutoff]\label{co:cutoff}Consider the exclusion process with fixed reservoir density $\rho\in(0,1)$ on a  sequence of networks  $(G_n)_{n\ge 1}$. Then the cutoff phenomenon
\begin{eqnarray}
\label{def:cutoff}
\forall \varepsilon\in(0,1),\quad \frac{\tmix^{(n)}(1-\varepsilon)}{\tmix^{(n)}(\varepsilon)} & \xrightarrow[n \to\infty]{} & 1,
\end{eqnarray}
occurs if and only if the sequence satisfies the so-called ``product condition'':
\begin{eqnarray}
\label{def:product}
\lambda_n\times \tmix^{(n)}(1/4) & \xrightarrow[n \to\infty]{} & +\infty.
\end{eqnarray}
Moreover, this remains valid if  $\rho$ varies with $n$, as long as it stays bounded away from $0$ and $1$.
\end{corollary}
The value $1/4$ appearing in (\ref{def:product}) is arbitrary, and can be replaced by any other precision $\varepsilon\in(0,1)$.
The interest of this criterion  is that it only involves  orders of magnitude: unlike the definition (\ref{def:cutoff}),  it can be checked without having to determine the precise pre-factor in front of mixing times. The product condition is well known to be necessary for cutoff, along any sequence of reversible chains \cite[Proposition 18.4]{MR3726904}. In the 2004 AIM workshop on mixing times, Peres \cite{peresamerican} conjectured that it is also sufficient. Unfortunately, counter-examples have been constructed; see \cite[Section 6]{chen2008cutoff} or \cite[Example 18.7]{MR3726904}. However, sufficiency has been established for all birth-and-death chains \cite{ding2010total} and, more generally, all random walks on trees \cite{MR3650406}. Corollary \ref{co:cutoff} adds the exclusion process with reservoirs to this short list. For general chains, finding an effective sufficient criterion for cutoff remains the most fundamental problem in the area of mixing times. See \cite{salez2021cutoff} for a first step in this direction.

\paragraph{Mixing time.} Corollary \ref{co:cutoff} can be used to predict the occurrence of a cutoff, but does not say \emph{where} it occurs. Fortunately, Theorem \ref{th:main} also implies sharp mixing-time estimates. We start with the following universal upper-bound.
\begin{corollary}[Upper-bound]\label{co:UB}There is a constant $c=c(\varepsilon,\rho)$, not depending on $G$, such that
\begin{eqnarray}
\label{generic}
\tmix(\varepsilon) & \le & \frac{\log |V|+c}{2\lambda}.
\end{eqnarray}
Moreover, this remains valid if the distance to equilibrium is measured in the stronger $L^2_\pi$ sense.
\end{corollary}
This result essentially asserts that the mixing time of the whole  system is at most that of $|V|$ independent random walkers. Proving a similar result for the conservative version of the model constitutes a long-standing open problem, see Oliveira \cite{MR3077529}, Alon and Kozma \cite{MR4164852}, or Hermon and Pymar \cite{MR4164461} for partial progress. The generic bound (\ref{generic}) happens to be sharp in many cases, as we will now see. Recall that $\psi$ denotes the eigenvector of $-\Delta$ corresponding to the smallest eigenvalue $\lambda$. By the Perron-Frobenius theorem,  $\psi$ is the only  eigenvector (up to scalar multiplication)  whose coordinates all have the same sign. The normalized vector $\overline{\psi}:=\psi/\langle \psi,\bf 1\rangle$  is  known as the \emph{quasi-stationary distribution} of the network $G$: it is the large-time limit of the distribution of a random walker on $G$ conditioned on not having been killed yet (see, e.g.,  \cite[Section 3.6.5]{aldous-fill-2014}). 

\begin{corollary}[Lower-bound]\label{co:LB} There is a constant $c=c(\varepsilon)$, not depending on $G,\rho$, such that
\begin{eqnarray}
\label{LB:gen}
\tmix(\varepsilon) & \ge & \frac{\log\left(\langle \psi,{\bf 1}\rangle^2\right)-c}{2\gamma}.
\end{eqnarray}
\end{corollary}
Note that the quantity $\langle \psi,{\bf 1}\rangle^2=\frac{1}{\|\overline{\psi}\|^2}$  measures how {balanced} the unit vector $\psi$ is: it ranges from $1$ (when $\psi$ concentrates on a single entry) to $|V|$ (when all entries of $\psi$ are equal). In particular, the  bounds  (\ref{generic}) and (\ref{LB:gen}) match as soon as $\psi$ is sufficiently \emph{delocalized}, in the following precise sense.
\begin{corollary}[Delocalization implies cutoff]\label{co:smooth} Consider the exclusion process with fixed reservoir density $\rho\in(0,1)$ on any  sequence of networks  $(G_n)_{n\ge 1}$ satisfying the delocalization condition
\begin{eqnarray*}
{\langle \psi_n,{\bf 1}\rangle^2}& \ge & {|V_n|}^{1-o(1)}.
\end{eqnarray*}
Then, cutoff occurs at the following time:
\begin{eqnarray}
\label{cutoff}
 \tmix^{(n)}(\varepsilon) & \underset{n\to\infty}{\sim} & \frac{\log |V_n|}{2\lambda_n}.
\end{eqnarray}
Here again, this remains valid if  $\rho$ varies with $n$, as long as it is bounded away from $0$ and $1$.
\end{corollary}
\paragraph{Other metrics. }Finally, let us briefly discuss what happens when the total-variation distance appearing in  (\ref{def:mixing}) is replaced by one of the stronger divergences mentioned above. First recall that, for any measure $\mu\in\rX$, we have
\begin{eqnarray*}
2\|\mu-\pi\|_{\textsc{tv}}^2 & \le & \dkl(\mu||\pi) \ \le \ \left\|\frac{\mu}{\pi}-1\right\|^2_{L^2_\pi},
\end{eqnarray*}
where the first inequality is Pinsker's inequality and the second is simply $\log u\le u-1$.
As a consequence, the total-variation cutoff announced in Corollary \ref{co:smooth} also occurs in  relative entropy and in the Hilbert norm. However, it turns out that when the distance to equilibrium is measured in separation distance or in the supremum norm, one has to wait twice longer. More precisely, our methods imply the following analogue of Theorem \ref{th:main}, in which $\|\zz(t)\|^2$ is essentially  replaced with $\|\zz(t/2)\|^2$. Note that the separation distance is,   always a lower bound on the supremum distance.

\begin{corollary}[Separation and supremum norm]\label{co:metrics}In  separation distance, we have the lower-bound
\begin{eqnarray*}
\sep\left({\rP_t(x_\star,\cdot)},\pi\right) & \ge & \frac{\left\|\zz\left(\frac t 2\right)\right\|^2}{1+\left\|\zz\left(\frac t 2\right)\right\|^2},
\end{eqnarray*}
for any $t\ge 0$. 
On the other hand,  the   supremum norm satisfies the  upper-bound 
\begin{eqnarray*}
\max_{x,y\in\rX}\left|\frac{\rP_t(x,y)}{\pi(y)}-1\right| & \le & {\exp\left(\frac{\left\|\zz\left(\frac t 2\right)\right\|^2}{\rho_\star}\right)-1}.
\end{eqnarray*}
In particular, the total-variation cutoff announced in Corollary \ref{co:smooth} also holds in separation distance and in the uniform norm, but without the $2$ in the denominator of (\ref{cutoff}).
\end{corollary}

\subsection{Examples}\label{sec:ex} To illustrate the above results, we now specialize them to two  important examples. The first one is the  network induced by an arbitrary box  $V=[n_1]\times\cdots\times [n_d]$ in the $d-$dimensional  lattice $\GG=\dZ^d$. We call this network the box of dimensions $n_1\times\cdots\times n_d$  with \emph{open boundaries}. Note that the special case $d=1$ is precisely the model studied in \cite{goncalves2021sharp}, for which (\ref{result:2}) was established.  The second example is the natural \emph{semi-open} variant where the ambient lattice $\GG=\dZ^d$ is replaced by the $d-$dimensional semi-lattice $\GG=\dN^d$, where $\dN=\{1,2,\ldots\}$. The special case $d=1$ is then precisely the one for which (\ref{result:1}) was established in   \cite{gantert2021mixing}. Our results  imply the following  high-dimensional generalization of those two results.

\begin{corollary}[Spectral gap, mixing time  and cutoff on boxes]\label{co:ex} For the exclusion process with reservoir density $\rho\in(0,1)$ on a box of dimension $n_1\times\cdots\times n_d$  with open boundaries, we have
\begin{eqnarray*}
\lambda & = & 2\sum_{k=1}^d\left[1-\cos\left(\frac{\pi}{n_k+1}\right)\right],\qquad \psi(i_1,\ldots,i_d) \ = \ \prod_{k=1}^d\sqrt{\frac{2}{n_k+1}}\sin\left(\frac{\pi i_k}{n_k+1}\right),
\end{eqnarray*}
while for the semi-open variant, we have
\begin{eqnarray*}
\lambda & = & 2\sum_{k=1}^d\left[1-\cos\left(\frac{\pi}{2n_k+1}\right)\right], \qquad \psi(i_1,\ldots,i_d) \ = \ \prod_{k=1}^d\sqrt{\frac{4}{2n_k+1}}\sin\left(\frac{\pi i_k}{2n_k+1}\right).
\end{eqnarray*}
Consequently, in both cases, the total-variation mixing time satisfies 
\begin{eqnarray*}
\frac{\log|V|-c d}{2\lambda} \ \le & \tmix(\varepsilon) & \le \ \frac{\log|V|+c}{2\lambda},
\end{eqnarray*}
where $c=c(\varepsilon,\rho)$ does not depend on $d,n_1,\ldots,n_d$. In particular,    there is cutoff at time (\ref{cutoff}) along any sequence of boxes $(G_n)_{n\ge 1}$ with diverging average logarithmic side-length, i.e. 
\begin{eqnarray*}
\frac{\log |V_n|}{d_n} & \xrightarrow[n\to\infty]{} & +\infty.
\end{eqnarray*}
For example, if $G_n$ denotes the box of dimension $n\times\cdots\times n$ with open boundaries, then
\begin{eqnarray*}
 \tmix^{(n)}(\varepsilon) & = &  \frac{n^2\log n }{2\pi^2}+\cO(n^2),
\end{eqnarray*}
where the implicit constant depends only on $\varepsilon,\rho\in(0,1)$. For the semi-open version, we obtain
\begin{eqnarray*}
 \tmix^{(n)}(\varepsilon) & = &  \frac{2n^2\log n }{\pi^2}+\cO(n^2).
\end{eqnarray*}
We emphasize that those estimates are valid even if the ambient dimension $d=d_n$ varies with $n$. The reservoir density $\rho$ is allowed to vary as well, as long as it is bounded away from $0$ and $1$. Moreover, a cutoff also occurs at the above times when the distance to equilibrium is measured in relative entropy or  Hilbert norm, and twice later when measured in  separation  or   uniform norm. 
\end{corollary}
One could also consider the hybrid case where the box has open boundaries in certain directions and semi-open boundaries in others. The above expressions for $\lambda,\psi$ adapt in the obvious way. For example, on a $n\times n$ grid with boundaries on three of the four sides, one finds
\begin{eqnarray*}
\tmix(\varepsilon) & = & \frac{4n^2\log n}{5\pi^2}+\cO(n^2).
\end{eqnarray*}
Finally, we emphasize that boxes are just a particular example chosen for comparison with the existing literature \cite{goncalves2021sharp,gantert2021mixing}: our general results will yield explicit asymptotics on any network for which we can compute or estimate the spectral statistics $\lambda$ and $\psi$. This is particularly easy in the case of Cartesian products, thanks to an obvious   tensorization property of $\Delta$ (see (\ref{tensor}) below).

\section{Perturbations of product measures}
\label{sec:pert}
Without loss of generality, we henceforth assume that the vertex set $V$ is $[n]=\{1,\ldots,n\}$. Theorem \ref{th:main} happens to be a special case of a new and general two-sided estimate, which is valid for all \emph{negatively dependent perturbations} of product measures.  We establish this general result in the present section, and will specialize it to the exclusion process with reservoirs in Section \ref{sec:app}.
\subsection{Framework}
 To what extent can a random binary vector be statistically distinguished from a perturbed version where a few coordinates have been modified?  To formalize this question, consider a random binary vector $X^\star=(X_1^\star,\ldots,X_n^\star)$ distributed according to the reference measure $\pi=\cB_\rho^{\otimes n}$, and let  $(Y,Z)$ be an arbitrary pair  of  random binary vectors, independent of $X^\star$. We think of $(Y,Z)$ as \emph{noise variables}, which we use to {perturb} $X^\star$ as follows: $Z$ indicates which coordinates of $X^\star$ are to be modified, and $Y$ specifies the values to be used for replacement. Specifically, for $i\in[n]$, we set
\begin{eqnarray}
\label{def:pert}
{X}_i & := & (1-Z_i)X_i^\star+Z_iY_i.
\end{eqnarray}
At least intuitively, the law $\mu$ of the perturbed vector ${X}=(X_1,\ldots,X_n)$ should be close to the reference law $\pi$ as long as the support of $Z$ (the perturbed region) is sufficiently \emph{small} and \emph{delocalized}, regardless of the noise  $Y$. Our goal here is to quantify this statement by a precise estimate on  the distance $\|{\mu}-\pi\|_{\textsc{tv}}$, as a function of the mean noise vector
\begin{eqnarray*}
\zz & := & \left(\EE[Z_1],\ldots,\EE[Z_n]\right).
\end{eqnarray*}
 We will see that, for a broad class of perturbations, the correct answer is given in a two-sided way by the Euclidean norm $\|\zz\|=\sqrt{\zz_1^2+\cdots+\zz_n^2}$.

\subsection{Upper-bound}
A naive way to quantify the impact of the perturbation consists in using the  Wasserstein bound
\begin{eqnarray}
\label{bound:W1}
\|{\mu}-\pi\|_{\textsc{tv}} & \le & \EE\left[\sum_{i=1}^n\left|{X}_{i}-X_i^\star\right|\right] \ \le \ \sum_{i=1}^n\EE[Z_i] \ = \  \langle \zz,\bf 1\rangle.
\end{eqnarray}
Albeit simple and general, this  estimate is too pessimistic, because it only focuses on the \emph{total size} of the perturbed region $Z$, and not on its \emph{localized/delocalized} nature in space:  intuitively, zeroing a fixed, deterministic entry should be much easier to detect (in total-variation distance) than zeroing a uniformly chosen entry. Yet, (\ref{bound:W1}) does not distinguish at all between those two situations. 

To get a feeling of how much better the answer could be for delocalized perturbations, let us investigate the elementary but instructive case where the coordinates of $Z$ are independent, while  $Y$ is deterministically equal to the extremal vector $x_\star$. Under this simplifying assumption, the resulting law ${\pi}$ is clearly a product measure, so an easy and classical computation yields
\begin{eqnarray}
\label{bound:L2prod}
\|{\mu}-\pi\|_{\textsc{tv}} & \le & \frac 12\left\|\frac{{\mu}}{\pi}-1\right\|_{L^2_\pi} \ = \   \frac 12\sqrt{\prod_{i=1}^n\left(1+\frac{1-\rho_\star}{\rho_\star} \zz_i^2\right)-1}
\end{eqnarray}
In the small-perturbation regime where  $\zz\to \bf 0$, the right-hand side decays like $\|\zz\|$, which constitutes a considerable improvement over the linear dependency predicted by (\ref{bound:W1}).  

Of course, the explicit computation of the $L^2_\pi-$norm in  (\ref{bound:L2prod}) crucially relies on independence, and getting sharp estimates beyond the product case constitutes a notoriously challenging task. For certain perturbations of spin systems, a sophisticated and powerful approach called \emph{information percolation} was developed by Lubetzky and Sly in an impressive series of papers \cite{MR3020173,MR3193965,MR3486171,MR3434254}. Here our  main -- and much more elementary -- finding is that  the trivial bound (\ref{bound:L2prod}) still holds if one assumes that the coordinates of $Z$ are \emph{negatively dependent} (ND) in the following  sense: 
\begin{eqnarray}
\label{assume:NA}
\forall S\subseteq [n],\quad 
\EE\left[\prod_{i\in S}Z_i\right] & \le & \prod_{i\in S}\EE[Z_i].
\end{eqnarray}
More precisely, we have the following universal upper-bound. 
\begin{lemma}[Upper-bound on the impact of ND perturbations]\label{lm:L2}If $Z$ is {\rm ND}, then \begin{eqnarray*}
\left\|\frac{{\mu}}{\pi}-1\right\|_{L^2_\pi} & \le & \sqrt{\prod_{i=1}^n\left(1+\frac{1-\rho_\star}{\rho_\star}\zz_i^2\right)-1}.
\end{eqnarray*}
\end{lemma}
\begin{proof} Since $X^\star$ is independent of $(Y,Z)$ and has law $\pi=\cB_\rho^{\otimes n}$, we have
\begin{eqnarray*}
\mu(y)  & = & \EE\left[\left(\prod_{i\notin Z}\cB_\rho(y_i)\right)\left(\prod_{i\in Z}{\bf 1}_{Y_i=y_i}\right)\right].
\end{eqnarray*}
Denoting by $(Y',Z')$ an independent copy of $(Y,Z)$, we deduce that
\begin{eqnarray*}
\mu^2(y) & = & \EE\left[\left(\prod_{i\notin Z}\cB_\rho(y_i)\right)\left(\prod_{i\notin Z'}\cB_\rho(y_i)\right)\left(\prod_{i\in Z}{\bf 1}_{Y_i=y_i}\right)\left(\prod_{i\in Z'}{\bf 1}_{Y_i'=y_i}\right)\right].
\end{eqnarray*}
Dividing through by $\pi(y)= \cB_\rho(y_1)\cdots\cB_\rho(y_n)$ and simplifying, we obtain
\begin{eqnarray*}
\frac{\mu^2(y)}{\pi(y)} & = & \EE\left[\left(\prod_{i\in Z\cap Z'}\frac{1}{\cB_\rho(y_i)}\right)\left(\prod_{i\notin Z\cup Z'}\cB_\rho(y_i)\right)\left(\prod_{i\in Z}{\bf 1}_{Y_i=y_i}\right)\left(\prod_{i\in Z'}{\bf 1}_{Y_i'=y_i}\right)\right]\\
& \le & \EE\left[\frac{1}{\rho_\star^{|Z\cap Z'|}}\left(\prod_{i\notin Z\cup Z'}\cB_\rho(y_i)\right)\left(\prod_{i\in Z}{\bf 1}_{Y_i=y_i}\right)\left(\prod_{i\in Z'\setminus Z}{\bf 1}_{Y_i'=y_i}\right)\right],
\end{eqnarray*}
because  $\min\left\{\cB_\rho(0),\cB_\rho(1)\right\}= \rho_\star$. Summing over all $y\in\{0,1\}^n$, we arrive at
\begin{eqnarray*}
\left\|\frac{\mu}{\pi}\right\|^2_{L^2_\pi} &\le & \EE\left[\frac{1}{\rho_\star^{|Z\cap Z'|}}\right].
\end{eqnarray*}
Finally, recalling that $Z,Z'$ are i.i.d. and satisfy (\ref{assume:NA}), we can write 
\begin{eqnarray*}
\EE\left[\frac{1}{\rho_\star^{|Z\cap Z'|}}\right] & = & \EE\left[\prod_{i=1}^n\left(1+\frac{1-\rho_\star}{\rho_\star}Z_iZ'_i\right)\right]\\
 & = & \sum_{S\subseteq[n]}\left(\frac{1-\rho_\star}{\rho_\star}\right)^{|S|}\EE\left[\prod_{i\in S}Z_iZ'_i\right] \\ 
& \le & \sum_{S\subseteq[n]}\left(\frac{1-\rho_\star}{\rho_\star}\right)^{|S|}\left(\prod_{i\in S}\EE\left[Z_i\right]\right)^2\\
& = &
\prod_{i=1}^n\left(1+\frac{1-\rho_\star}{\rho_\star}\zz_i^2\right).
\end{eqnarray*}
Since $\left\|\frac{\mu}{\pi}-1\right\|^2_{L^2_\pi}=\left\|\frac{\mu}{\pi}\right\|^2_{L^2_\pi}-1$, the claim is proved.
\end{proof}
\begin{remark}[Sharpness]As already explained (or deduced from a careful examination of the above proof), the inequality in Lemma \ref{lm:L2} is an equality when the coordinates of $Z$ are independent, while  $Y$ is deterministically set to $x_\star$. Thus, Lemma \ref{lm:L2} is sharp, and states that the $L^2$ impact of negatively-dependent perturbations is maximized in the product case.
\end{remark}

We now complement the above upper-bound with a matching lower-bound.
\subsection{Lower-bound}
Lemma \ref{lm:L2} states that $\mu$ is uniformly close to $\pi$ (even in the strong $L^2_\pi$ sense) whenever  $\|\zz\|$ is small. Conversely, we  now show that $\mu$ is uniformly \emph{far} from $\pi$ whenever $\|\zz\|$ is large, at least in the extreme case $Y=x_\star$. We  here only need  the special case $|S|=2$ of the \rm{ND} condition (\ref{assume:NA}), namely 
\begin{eqnarray}
\label{assume:NC}
\forall i\ne j,\quad \cov(Z_i,Z_j) & \le & 0.
\end{eqnarray}
\begin{lemma}[Matching lower bound]\label{lm:L1}Assume that $Z$ satisfies (\ref{assume:NC}), and that $Y=x_\star$. Then,
\begin{eqnarray*}
\|{\mu}-\pi\|_{\textsc{tv}} & \ge & \frac{\|\zz\|^2}{4+\|\zz\|^2}.
\end{eqnarray*}
\end{lemma}
\begin{proof}
A simple way to bound total variation from below consists in applying the general  inequality
\begin{eqnarray}
\label{distinguishing}
\|{\mu}-\pi\|_{\textsc{tv}}   & \ge & \frac{(\EE[f(X)]-\EE[f(X^\star)])^2}{(\EE[f(X)]-\EE[f(X^\star)])^2+2\mathrm{Var}[f(X)]+2\mathrm{Var}\left[f(X^\star)\right]},
\end{eqnarray}
to an appropriate \emph{distinguishing statistics} $f\colon \{0,1\}^n\to\dR$, 
see \cite[Proposition 7.8]{MR3726904}. Here we choose
\begin{eqnarray*}
f(x) & := & \sum_{i=1}^n \zz_i(x_i-\rho),
\end{eqnarray*}
where we recall that $\zz_i:=\EE[Z_i]$.
Since the coordinates of $X^\star$ are i.i.d. with law $\cB_\rho$, we  have
\begin{eqnarray*}
\EE[f(X^\star)] & = & 0; \\
\mathrm{Var}[f(X^\star)] & = & \rho(1-\rho) \|\zz\|^2\ \le \ \frac{\|\zz\|^2}{4}.
\end{eqnarray*}
On the other hand, we know that ${X}_i=(1-Z_i)X_i^\star+Z_i{{\bf 1}_{\rho<1/2}}$ with $X^\star,Z$ independent. Thus, $\EE[X_i]=\rho+\zz_i({\bf 1}_{\rho<1/2}-\rho)$, $\cov({X}_i,{X}_j)=(1-\rho_\star)^2\cov(Z_i,Z_j)\le 0$ for $i\ne j$, and
\begin{eqnarray*}
\EE^2[f({X})] & = & \|\zz\|^4({\bf 1}_{\rho<1/2}-\rho)^2 \ \ge \ \frac{\left\|\zz\right\|^4}{4};\\
\mathrm{Var}[f({X})] & \le & \sum_{i=1}^n\zz_i^2\mathrm{Var}({X}_i) \ \le \ \frac{\|\zz\|^2}{4}.
\end{eqnarray*}
Inserting those estimates into (\ref{distinguishing}) readily yields the claimed bound. 
\end{proof}

\section{Application to   exclusion  with reservoirs}
\label{sec:app}
We now show that our general perturbation theory for negatively-dependent measures applies to the special case of the exclusion process with reservoirs. More precisely,
\begin{enumerate}
\item We show in Section \ref{sec:graphical} that at any given time $t\ge 0$ and from any given initial state $x\in\rX$, the distribution $\mu=\rP_t(x,\cdot)$ of the system  is  a perturbation of the equilibrium measure $\pi$ in the sense of Definition (\ref{def:pert}), for a certain noise vector $(Y,Z)$ that we explicitate.
\item We show in Section \ref{sec:SR} that the perturbed region $Z$ is negatively-dependent in the sense of Assumption (\ref{assume:NA}), and that its marginals satisfy (\ref{death}), thereby implying Theorem \ref{th:main}.
\item Finally, in Section \ref{sec:coro}, we  detail the arguments that lead from Theorem \ref{th:main} to Corollaries \ref{co:gap}-\ref{co:ex}.
\end{enumerate}

\subsection{Graphical construction}
\label{sec:graphical}
Let  $\bX=(X(t))_{t\ge 0}$ be an exclusion process with reservoir density $\rho$ on $G$, starting from an arbitrary  state $x\in\rX$. The \emph{graphical construction} provides a standardized representation of 
$\bX$ in the form
\begin{eqnarray}
\label{def:X}
\bX & := & \Psi\left(x,(\Xi_i)_{1\le i\le n},(\Xi_{i j})_{1\le i<j\le n},\left(\xi_{k}\right)_{k\ge 1}\right),
\end{eqnarray}
where the variables $\Xi_1,\ldots,\Xi_n,\Xi_{11},\ldots\Xi_{nn},\xi_{1},\xi_2,\ldots$ are independent and as follows:
\begin{itemize}
\item $\Xi_{ij}$ is a Poisson point process of rate $c(i,j)$ specifying the exchange times between  $i$ and $j$.
\item $\Xi_i$ is  a Poisson point process of rate $\kappa(i)$ specifying the resampling times at site $i$.
\item $\xi_k$  is a $\cB_\rho-$variable specifying the new value to be assigned when the $k-$th resampling occurs.
\end{itemize}
From this data, the trajectory $\bX=(X(t))_{t\ge 0}$ is deterministically obtained as  the unique right-continuous, $\{0,1\}^n-$valued function which equals $x$ at time $0$, is constant outside the locally finite set $\Xi:=\Xi_1\cup\ldots\cup\Xi_n\cup \Xi_{11}\cup \ldots\cup \Xi_{nn}$, and jumps at  $t\in \Xi$ as follows:
\begin{enumerate}[(i)]
\item if $t\in\Xi_{ij}$, then $X(t)=\left(X(t^-)\right)^{i\leftrightarrow j}$.
\item if $t\in\Xi_i$ and $t$ is the $k-$th smallest point in $\Xi_1\cup\ldots\cup\Xi_n$, then $X(t)=\left(X(t^-)\right)^{i,\xi_k}$.
\end{enumerate}
This rigorously specifies the measurable map $\Psi$ appearing in (\ref{def:X}). We may now couple $\bX$ with a stationary process $\bX^\star$ by setting
\begin{eqnarray}
\bX^\star & := & \Psi\left(\zeta,(\Xi_i)_{1\le i\le n},(\Xi_{i,j})_{1\le i<j\le n},\left(\xi_{k}\right)_{k\ge 1}\right),
\end{eqnarray}
where $\zeta$ denote a $\pi-$distributed random variable independent of $\Xi_1,\ldots,\Xi_n,\Xi_{11},\ldots\Xi_{nn},\xi_{1},\xi_2,\ldots$. In order to compare $\bX$ and $\bX^\star$, we introduce two  auxiliary processes:
\begin{eqnarray*}
\bY & := & \Psi\left(x,(\Xi_i)_{1\le i\le n},(\Xi_{i,j})_{1\le i<j\le n},{\bf 0}\right);\\
\label{def:Z}\bZ & := & \Psi\left({\bf 1},(\Xi_i)_{1\le i\le n},(\Xi_{i,j})_{1\le i<j\le n},{\bf 0}\right).
\end{eqnarray*}
Note that  $\bY,\bZ$ are  exclusion processes with reservoir density $\rho=0$ starting from   $x$ and $\bf 1$, respectively.
The next lemma shows that at any time $t\ge 0$, the random vector $X(t)$ is the perturbation of $X^\star(t)$ induced by the noise $(Y(t),Z(t))$, in the sense of Definition (\ref{def:pert}).
\begin{lemma}[Perturbative structure of the exclusion process]At any time $t\ge 0$, the random vector $X^\star(t)$ is $\pi-$distributed and independent of $\left(\bY,\bZ\right)$. Moreover, we have
\begin{eqnarray}
\label{perturbation}
\forall i\in[n],\qquad X_i(t) & = & \left(1-Z_i(t)\right)X_i^\star(t)+Z_i(t)Y_i(t).
\end{eqnarray}
Finally, in this formula, we can replace $Y_i(t)$ by $1$ when $x=\bf 1$, and by $0$ when $x=\bf 0$.
\end{lemma}
\begin{proof}
The law $\pi=\cB_\rho^{\otimes n}$ is trivially preserved under swapping  two coordinates or replacing a coordinate with a fresh $\cB_\rho-$distributed variable. Thus,  the  conditional law of $X^\star(t)$  given the point processes $\Xi_1,\ldots,\Xi_n,\Xi_{11},\ldots\Xi_{nn}$ is $\pi$. Since $\bY,\bZ$ are measurable functions of $\Xi_1,\ldots,\Xi_n,\Xi_{11},\ldots\Xi_{nn}$, the first claim follows. For the second, it suffices to note that  the identity (\ref{perturbation}) holds at time $t=0$ (both sides being equal to $x_i$) and is preserved at each discontinuity time $t\in\Xi$ (in case (i) above,  the transposition ${i\leftrightarrow j}$ is simultaneously applied to the four  vectors $X^\star(t),X(t),Y(t),Z(t)$, and in case (ii) we have $Z_i(t)=0$ and $X_i(t)=X_i^\star(t)=\xi^k$). Finally,  observe that the process $\bY$ is equal to $\bZ$ in the case $x=\bf 1$, and to $\bf 0$  in the case $x=\bf 0$. 
\end{proof}
\begin{remark}[Strong stationary time]\label{rk:sst}The above lemma readily implies that the random variable
$
T  :=  \inf\left\{t\ge 0\colon \bZ(t)={\bf 0}\right\}
$
is a \emph{strong stationary time}, i.e. 
\begin{eqnarray*}
\forall (t,y)\in\dR_+\times \rX,\qquad \PP(X(t)=y,T\le t) & = & \pi(y)\PP(T\le t).
\end{eqnarray*}
In particular, this classically implies the separation-distance bound
\begin{eqnarray*}
\forall t\in\dR_+,\qquad \sep\left(\rP_t(x,\cdot),\pi\right) & \le & \PP(T>t),
\end{eqnarray*}
with equality in the extreme cases $x={\bf 0}$ and $x={\bf 1}$; see \cite[Lemma 6.12 and Proposition 6.14]{MR3726904}. 
\end{remark}
\subsection{Analysis of the perturbed region}
\label{sec:SR}
In order to apply the  results of Section \ref{sec:pert}, we must verify that at any time $t\ge 0$, the perturbed region $Z(t)$  meets our  negative dependence (\rm{ND}) requirement (\ref{assume:NA}). 
By construction, the process $\bZ=(Z(t))_{t\ge 0}$ is an exclusion process with reservoir density $\rho=0$, starting from $Z(0)=\bf 1$. Thus, the claim is a special case of the following general result which, for the conservative variant of the model, was established by Liggett \cite[ Proposition VIII.1.7]{MR2108619} (see Borcea, Br\"{a}nd\'{e}n and Liggett \cite[Proposition 5.1]{MR2476782} for a considerable refinement). 
\begin{lemma}[Negative dependence for exclusion with reservoirs]\label{lm:ND}Let $\bX=(X(t))_{t\ge 0}$ denote an exclusion process with reservoir density $\rho\in[0,1]$ on an arbitrary network $G$, and suppose that the initial random vector $X(0)$ is ND. Then, so is $X(t)$ for all $t\ge 0$.
\end{lemma}
\begin{proof}
Recall that the law $\mu(t)$ of $X(t)$ is given by $\mu(t)=\mu(0)e^{t\rL}$, where $\rL$ is the generator defined in the introduction. Thus, we want to show that $\rL$ is ND-preserving in the following sense:
\begin{eqnarray*}
\mu \textrm{ is ND} & \Longrightarrow & \left(\forall t\ge 0,\  \mu e^{t\rL}\textrm{ is ND}\right).
\end{eqnarray*}
Here is a simple observation that will substantially reduce our task: if $\rL_1,\rL_2$  are two ND-preserving generators on $\{0,1\}^n$, then so is any superposition of the form $\rL=\lambda_1\rL_1+\lambda_2\rL_2$ with $\lambda_1,\lambda_2\ge 0$. Indeed, Trotter product formula \cite[p. 33]{MR838085} asserts that for all $t\ge 0$, 
\begin{eqnarray*}
e^{t(\lambda_1\rL_1+\lambda_2\rL_2)} & = & \lim_{k\to\infty}\left(e^{\frac{\lambda_1 t}{k}\rL_1}e^{\frac{\lambda_2 t}{k}\rL_2}\right)^k, 
\end{eqnarray*}
and the claim follows because the ND property is preserved under weak convergence. Consequently, we only need to separately prove Lemma \ref{lm:ND} in the following three elementary cases:
\begin{enumerate}[(i)]
\item $\rL f(x)=f(x^{i,1})-f(x)$ (creation at $i$)
\item $\rL f(x)=f(x^{i,0})-f(x)$ (annihilation at $i$)
\item $\rL f(x)=f(x^{i\leftrightarrow j})-f(x)$ (exchange between $i$ and $j$)
\end{enumerate}
To do so, we suppose that $X(0)$ is ND.  In case (i)-(ii), we have the representation
\begin{eqnarray*}
X(t) & = & \left\{
\begin{array}{ll}
X(0)^{i,b} & \textrm{if }\cN(t)\ge 1\\
X(0) & \textrm{else},
\end{array}
\right.
\end{eqnarray*}
  where  $(\cN(t))_{t\ge 0}$ denotes  a unit-rate Poisson process  independent of $X(0)$, and with $b=1$ in the creation case and $b=0$ in the annihilation case (ii). Thus, the desired inequality
\begin{eqnarray}
\label{goal}
\EE\left[\prod_{k\in S}X_k(t)\right] & \le &  \prod_{k\in S}\EE\left[X_k(t)\right]
\end{eqnarray}
is trivially satisfies if the set $S$ does not contain $i$. On the other hand, if $S$  contains $i$, then
 \begin{eqnarray*}
\EE\left[\prod_{k\in S}X_k(t)\right] & = & b(1-e^{-t})
\EE\left[\prod_{k\in S\setminus\{i\}}X_k(0)\right]+e^{-t}\EE\left[\prod_{k\in S}X_k(0)\right]\\
& \le & b (1-e^{-t})  \prod_{k\in S\setminus\{i\}}\EE[X_k(0)]+e^{-t}\prod_{k\in S}\EE[X_k(0)]\\
& = & \prod_{k\in S}\EE[X_k(t)],
\end{eqnarray*}
as desired. In case (iii), we have
\begin{eqnarray*}
X(t) & = & \left\{
\begin{array}{ll}
X(0)^{i\leftrightarrow j} & \textrm{if }\cN(t)\textrm{ is odd}\\
X(0) & \textrm{else} \\
\end{array}
\right.
\end{eqnarray*}
In particular, (\ref{goal}) trivially holds if $S$ contains neither $i$ nor $j$. On the other hand, if $S$ contains $i$ but not $j$ (or vice versa), then writing $\theta_t=\PP(\cN(t)\textrm{ is even})$, we have
\begin{eqnarray*}
\EE\left[\prod_{k\in S}X_k(t)\right] & = & \theta_t
\EE\left[X_j(0)\prod_{k\in S\setminus\{i\}}X_k(0)\right]+ (1-\theta_t)
\EE\left[\prod_{k\in S}X_k(0)\right] \\
& \le & \theta_t\EE[X_j(0)]\prod_{k\in S\setminus\{i\}}\EE\left[X_k(0)\right]+ (1-\theta_t)
\prod_{k\in S}\EE\left[X_k(0)\right] \\
& = & \prod_{k\in S}\EE\left[X_k(t)\right],
\end{eqnarray*}
as desired. Finally, if $S$ contains both $i$ and $j$, then 
\begin{eqnarray*}
\EE\left[\prod_{k\in S}X_k(t)\right] & = & \EE\left[\prod_{k\in S}X_k(0)\right] \ \le \ \prod_{k\in S}\ \EE[X_k(0)] \ = \ \ \EE[X_i(0)]\EE[X_j(0)]\prod_{k\in S\setminus\{i,j\}}\ \EE[X_k(0)],
\end{eqnarray*}
so, the desired inequality (\ref{goal}) boils down to 
$
\EE[X_i(0)]\EE[X_j(0)] \le \EE[X_i(t)]\EE[X_j(t)]
$. But since
\begin{eqnarray*}
\EE[X_i(t)] & = & (1-\theta_t)\EE[X_i(0)]+\theta_t\EE[X_j(0)]\\
\EE[X_j(t)] & = & (1-\theta_t)\EE[X_j(0)]+\theta_t\EE[X_i(0)],
\end{eqnarray*}
 the claim further simplifies to  
$
\theta_t(1-\theta_t)\left(\EE[X_i(0)]-\EE[X_j(0)]\right)^2  \ge  0,
$
which clearly holds.
\end{proof}
To establish Theorem \ref{th:main}, it now only remains to prove that the  marginals $\zz_i(t):=\EE[Z_i(t)]$ satisfy the differential equation (\ref{death}). This is exactly the special case $\rho=0,X(0)={\bf 1}$ of the following  result. 
\begin{lemma}[Single-site marginals]Let $\bX=(X(t))_{t\ge 0}$ denote an exclusion process with reservoir density $\rho\in[0,1]$ on an arbitrary network $G$. Then, the mean function $\zz\colon t\mapsto\EE[X(t)]$ solves
\begin{eqnarray*}
\frac{{\rm d}\zz}{{\rm d}t} & = & \Delta(\zz-\rho).
\end{eqnarray*}
\end{lemma}
\begin{proof}Dynkin's formula asserts that for any observable $f\colon\rX\to\dR$,  
\begin{eqnarray*}
\frac{{\rm d} }{{\rm d}t}\EE[f(X(t))] & = & \EE\left[(\rL f)(X(t))\right].
\end{eqnarray*}
Now, for the $i-$th projection $f(x)=x_i$, we readily compute
\begin{eqnarray}
\label{reuse}
\forall x\in\rX,\quad (\rL f)(x) & = & -\kappa(i)\left(x_i-\rho\right)+\sum_{j=1}^nc(i,j)\left(x_j-x_i\right).
\end{eqnarray}
Since the right-hand side is precisely the $i-th$ coordinate of the vector $\Delta(x-\rho)$, we are done.
\end{proof}

\subsection{Putting things together}
\label{sec:coro}
We now have all the ingredients needed to prove the results announced in the introduction.
\begin{proof}[Proof of Theorem \ref{th:main}]Consider the perturbation $(X^\star(t),X(t),Y(t),Z(t))$  defined in Section \ref{sec:graphical} and analyzed in Section \ref{sec:SR}. For this perturbation, Lemma \ref{lm:L2} reads
\begin{eqnarray*}
\left\|\frac{\rP_t(x,\cdot)}{\pi}-1\right\|_{L^2_\pi} & \le & \sqrt{\prod_{i=1}^n\left(1+\frac{1-\rho_\star}{\rho_\star}\zz_i^2\right)-1}\\
& \le & \sqrt{\exp\left(\frac{1-\rho_\star}{\rho_\star}\|\zz\|^2\right)-1}\\
& \le & \sqrt{\exp\left(\frac{\|\zz\|^2}{\rho_\star}\right)-1}.
\end{eqnarray*}
Since $x\in\rX$ is arbitrary, the upper-bound is proved. The lower bound is precisely Lemma \ref{lm:L1} .
\end{proof}
\begin{proof}[Proof of Corollary \ref{co:gap}] The spectral expansion (\ref{zz:key}) implies 
\begin{eqnarray}
\label{re:use}
\forall t\ge 0,\qquad \|\zz(t)\|^2 & \le & |V| e^{-2\lambda t}.
\end{eqnarray}
Inserting this into the upper-bound provided by Theorem \ref{th:main} and letting $t\to\infty$, we see that
\begin{eqnarray*}
\max_{x\in\rX}\left\|\frac{\rP_t(x,\cdot)}{\pi}-1\right\|_{L^2_\pi} & \le & e^{-\lambda t+o(t)}.
\end{eqnarray*}
 Since the spectral gap $\gamma(\rL)$ of a reversible generator $\rL$ coincides with the asymptotic exponential decay rate of the distance to equilibrium, we deduce that $\gamma(\rL)  \ge  \lambda$. Conversely, consider the function $f\colon\rX\to\dR$ defined by $f(x)  :=    \langle \psi,x-\rho\rangle$ for all $x\in\rX$. Recalling  (\ref{reuse}), we  have
\begin{eqnarray*}
(\rL f)(x) & = & \langle \psi,\Delta(x-\rho)\rangle\\
& = & \langle \Delta\psi,x-\rho\rangle\\
& = & -\lambda \langle \psi,x-\rho\rangle\\
& = & -\lambda f(x),
\end{eqnarray*}
where the second line uses the symmetry of $\Delta$. Thus,  $-\lambda$ is an eigenvalue of $\rL$, hence $\gamma(\rL)  \le  \lambda$.
\end{proof}
\begin{proof}[Proof of Corollary \ref{co:width}]Fix $\varepsilon\in(0,\frac 12)$ and set $t:=\tmix(1-\varepsilon)$. By the lower-bound in Theorem \ref{th:main},
\begin{eqnarray*}
\frac{\|\zz(t)\|^2}{4+\|\zz(t)\|^2} &  \le & 1-\varepsilon.
\end{eqnarray*}
This  implies  that $\|\zz(t)\|^2   \le  {4}/{\varepsilon}$, so the upper-bound in Theorem \ref{th:main}  yields
\begin{eqnarray*}
\max_{x\in\rX}\left\|\frac{\rP_t(x,\cdot)}{\pi}-1\right\|_{L^2_\pi}  &  \le & \exp\left(\frac{2}{\varepsilon\rho_\star}\right).
\end{eqnarray*}
Since the spectral gap $\lambda$ of a reversible generator $\rL$ coincides with the exponential contraction rate of the ${L^2_\pi}-$distance to equilibrium, we deduce that for all $s\ge 0$, 
\begin{eqnarray*}
\max_{x\in\rX}\left\|\frac{\rP_{t+s}(x,\cdot)}{\pi}-1\right\|_{L^2_\pi}  &  \le & \exp\left(\frac{2}{\varepsilon\rho_\star}-\lambda s\right).
\end{eqnarray*}
Choosing $s=\frac{3}{ \lambda\varepsilon\rho_\star}$ makes the right-hand side less than $\varepsilon$. Recalling the Cauchy-Schwarz inequality
\begin{eqnarray}
\label{CS}
\|\mu-\pi\|_{\textsc{tv}} & \le  & \frac{1}{2}\left\|\frac{\mu}{\pi}-1\right\|_{L^2_\pi},
\end{eqnarray}
valid for any probability measure $\mu$ on $\rX$, we conclude that  $\tmix(\varepsilon)  \le  t+s$. This is exactly the claim, with $c:=3/(\varepsilon\rho_\star)$
\end{proof}
\begin{proof}[Proof of Corollary \ref{co:cutoff}]
It is classical that the product condition is necessary for cutoff, see \cite[Proposition 18.4]{MR3726904}. Conversely, if the product condition holds, then by Corollary \ref{co:width} we have
\begin{eqnarray*}
\tmix^{(n)}(1/4) & \gg & \tmix^{(n)}(\varepsilon)-\tmix^{(n)}(1-\varepsilon),
\end{eqnarray*}
for any fixed $\varepsilon\in(0,\frac{1}{4})$, which precisely mean that there is cutoff.
\end{proof}
\begin{proof}[Proof of Corollary \ref{co:UB}]Fix $c\ge 0$ and set $t:=\frac{\log |V|+c}{2\lambda}$.    In view of (\ref{re:use}), we have
$
\|\zz(t)\|^2 \le e^{-c}.
$ Consequently, the upper-bound in Theorem \ref{th:main} implies
\begin{eqnarray*}
\max_{x\in\rX}\left\|\frac{\rP_{t}(x,\cdot)}{\pi}-1\right\|_{L^2_\pi}  &  \le & \sqrt{\exp\left(\frac{e^{-c}}{\rho_\star}\right)-1}.
\end{eqnarray*}
Choosing $c=c(\varepsilon,\rho)$  such that the right-hand side  equals $\varepsilon$ concludes the proof.
\end{proof}
\begin{proof}[Proof of Corollary \ref{co:LB}]Fix $\varepsilon\in(0,1)$ and set $t=\tmix(\varepsilon)$. By the lower bound in Theorem \ref{th:main}, we have
\begin{eqnarray*}
\frac{\|\zz(t)\|^2}{4+\|\zz(t)\|^2} &  \le &  \varepsilon,
\end{eqnarray*}
which implies 
$
\|\zz(t)\|^2  \le  \frac{4}{1-\varepsilon}.
$
On the other hand,  we have
$
\|\zz(t)\|^2   \ge \langle \psi,{\bf 1}\rangle^2 e^{-2\lambda t}
$, by (\ref{zz:key}). Combining these two inequalities, we deduce that
\begin{eqnarray*}
t & \ge & \frac{1}{2\lambda}\log\left(\frac{\langle\psi,{\bf 1}\rangle^2(1-\varepsilon)}4\right).
\end{eqnarray*}
This is precisely the claim, with $c:=\log\frac{4}{1-\varepsilon}$.
\end{proof}
\begin{proof}[Proof of Corollary \ref{co:smooth}]The claim readily follows from  Corollaries \ref{co:UB} and \ref{co:LB}.
\end{proof}
\begin{proof}[Proof of Corollary \ref{co:metrics}]An easy and classical consequence of reversibility (see \cite[p. 120]{MR2341319}) is that
\begin{eqnarray*}
\left|\frac{\rP_{t}(x,y)}{\pi(y)}-1\right| & \le & \left\|\frac{\rP_{t/2}(x,\cdot)}{\pi(\cdot)}-1\right\|_{L^2_\pi}\left\|\frac{\rP_{t/2}(y,\cdot)}{\pi(\cdot)}-1\right\|_{L^2_\pi},
\end{eqnarray*}
for all $t\ge 0$ and $x,y\in\rX$. Thus, the claimed upper-bound follows from Theorem \ref{th:main}. We now turn to the lower-bound from $x=x_\star$. Using the short-hand $S:=\langle Z(t),{\bf 1}\rangle$, we have by Remark \ref{rk:sst},
\begin{eqnarray*}
\sep\left(\rP_{t}(x_\star,\cdot),\pi\right) & = & \PP\left(S>0\right)\\
& \ge & \frac{\EE^2\left[S\right]}{{\rm Var}\left(S\right)+\EE^2\left[S\right]}\\
& \ge & \frac{\EE\left[S\right]}{1+\EE\left[S\right]},
\end{eqnarray*}
where   the second line uses Cantelli's one-sided improvement of Chebychev's inequality, and  the third the fact that ${\rm Var}\left(S\right)\le \EE[S]$ when $S$ is a sum of negatively correlated Bernoulli random variables. To conclude, it only remains to note that
\begin{eqnarray*}
\EE[S] & = & \langle \zz(t),{\bf 1}\rangle \ = \ \|\zz(t/2)\|^2,
\end{eqnarray*}
because $\zz(t)=e^{t\Delta}{\bf 1}$ and $\Delta$ is symmetric.
\end{proof}
\begin{proof}[Proof of Corollary \ref{co:ex}]By the Perron-Frobenius theorem, $\psi$ is characterized as the only eigenvector of $\Delta$ all of whose coordinates have the same sign. Thus, it is enough to check that the formula for $\psi$ proposed in the claim defines an eigenvector. This is well known (and immediate to check) in the case $d=1$. The general case follows by observing that the Laplace matrix $\Delta=\Delta^{(n_1,\ldots,n_d)}$ of a box of dimensions $n_1\times \ldots\times n_d$ (either with open, or with semi-open boundaries) tensorizes as follows:
\begin{eqnarray}
\label{tensor}
\Delta^{(n_1,\ldots,n_d)} & = & \Delta^{(n_1)}\oplus\cdots\oplus  \Delta^{(n_d)},
\end{eqnarray}
where $\oplus$ denotes the Kronecker sum of matrices. The rest of the claim is a straightforward application of our general results listed above.
\end{proof}
\paragraph{Acknowledgment.} This work was partially supported by Institut Universitaire de France. 

\bibliographystyle{plain}
\bibliography{draft}

\end{document}